\documentclass{svmult}

\usepackage{makeidx}
\usepackage{graphicx}
\usepackage{multicol}
\usepackage[bottom]{footmisc}

\usepackage{amsmath,amssymb}

\makeindex

\begin{document}

\title*{Higher-Order Calculus of Variations\\ on Time Scales}

\author{Rui A. C. Ferreira \and Delfim F. M. Torres}

\institute{Department of Mathematics, University of Aveiro,
3810-193 Aveiro, Portugal\\
\texttt{ruiacferreira@yahoo.com}, \texttt{delfim@ua.pt}}

\maketitle


\begin{abstract}
We prove a version of the Euler-Lagrange equations for certain
problems of the calculus of variations on time scales with
higher-order delta derivatives.

\bigskip


\bigskip

\end{abstract}


\section{Introduction}

Calculus of variations on time scales (we refer the reader to
Section~\ref{sec:2} for a brief introduction to time scales) has
been introduced in 2004 in the papers by Bohner
\cite{CD:Bohner:2004} and Hilscher and Zeidan \cite{zeidan}, and
seems to have many opportunities for application in economics
\cite{Atici06}. In both works of Bohner and Hilscher\&Zeidan, the
Euler-Lagrange equation for the fundamental problem of the
calculus of variations on time scales,
\begin{equation}
\label{eq:EL:B}
\mathcal{L}[y(\cdot)]=\int_{a}^{b}L(t,y^\sigma(t),y^\Delta(t))\Delta t
\longrightarrow\min, \quad \mbox{$y(a)=y_a$, $y(b)=y_b$},
\end{equation}
is obtained (in \cite{zeidan} for a bigger class of admissible
functions and for problems with more general endpoint conditions).
Here we generalize the previously obtained Euler-Lagrange equation
for variational problems involving delta derivatives of more than
the first order, \textrm{i.e.} for \emph{higher-order problems}.

We consider the following extension to problem \eqref{eq:EL:B}:
\begin{equation}
\label{problema }
\begin{gathered}
\mathcal{L}[y(\cdot)]=\int_{a}^{\rho^{r-1}(b)}
L(t,y^{\sigma^r}(t),y^{\sigma^{r-1}\Delta}(t),\ldots,
y^{\sigma\Delta^{r-1}}(t),y^{\Delta^r}(t))\Delta t\longrightarrow\min, \\
y(a)=y_a,  \ \  y\left(\rho^{r-1}(b)\right)=y_b ,\\
\vdots\\ \tag{P} y^{\Delta^{r-1}}(a)=y_a^{r-1},\ \
y^{\Delta^{r-1}}\left(\rho^{r-1}(b)\right)=y_b^{r-1},
\end{gathered}
\end{equation}
where $y^{\sigma^i\Delta^{r-i}}(t)\in\mathbb{R}^n,\
i\in\{0,\ldots,r\}$, $n,\ r\in\mathbb{N}$, and $t$ belongs to a
time scale $\mathbb{T}$. Assumptions on the time scale
$\mathbb{T}$ are stated in Section~\ref{sec:2}; the conditions
imposed on the Lagrangian $L$ and on the admissible functions $y$
are specified in Section~\ref{secprinci}. For $r=1$ problem
(\ref{problema }) is reduced to (\ref{eq:EL:B}); for $\mathbb{T} =
\mathbb{R}$ we get the classical problem of the calculus of
variations with higher-order derivatives.

While in the classical context of the calculus of variations,
\textrm{i.e.} when $\mathbb{T} = \mathbb{R}$, it is trivial to
obtain the Euler-Lagrange necessary optimality condition for
problem (\ref{problema }) as soon as we know how to do it for
(\ref{eq:EL:B}), this is not the case on the time scale setting.
The Euler-Lagrange equation obtained in
\cite{CD:Bohner:2004,zeidan} for (\ref{eq:EL:B}) follow the
classical proof, substituting the usual integration by parts
formula by integration by parts for the delta integral
(Lemma~\ref{integracao partes}). Here we generalize the proof of
\cite{CD:Bohner:2004,zeidan} to the higher-order case by
successively applying the delta-integration by parts and thus
obtaining a more general delta-differential Euler-Lagrange
equation. It is worth to mention that such a generalization poses serious technical difficulties and that the obtained necessary optimality condition is not true on a general time scale, being necessary some restrictions on $\mathbb{T}$. Proving an Euler-Lagrange necessary optimality condition for a completely arbitrary time scale $\mathbb{T}$ is a deep and difficult open question.

The paper is organized as follows: in Section~\ref{sec:2} a brief
introduction to the calculus of time scales is given and some
assumptions and basic results provided. Then, under the assumed
hypotheses on the time scale $\mathbb{T}$, we obtain in
Section~\ref{secprinci} the intended higher-order
delta-differential Euler-Lagrange equation.


\section{Basic Definitions and Results on Time Scales}
\label{sec:2}

A nonempty closed subset of $\mathbb{R}$ is called a \emph{time
scale} and it is denoted by $\mathbb{T}$.

The \emph{forward jump operator}
$\sigma:\mathbb{T}\rightarrow\mathbb{T}$ is defined by
$$\sigma(t)=\inf{\{s\in\mathbb{T}:s>t}\},
\mbox{ for all $t\in\mathbb{T}$},$$
while the \emph{backward jump operator}
$\rho:\mathbb{T}\rightarrow\mathbb{T}$ is defined by
$$\rho(t)=\sup{\{s\in\mathbb{T}:s<t}\},\mbox{ for all
$t\in\mathbb{T}$},$$ with $\inf\emptyset=\sup\mathbb{T}$
(\textrm{i.e.} $\sigma(M)=M$ if $\mathbb{T}$ has a maximum $M$)
and $\sup\emptyset=\inf\mathbb{T}$ (\textrm{i.e.} $\rho(m)=m$ if
$\mathbb{T}$ has a minimum $m$).

A point $t\in\mathbb{T}$ is called \emph{right-dense},
\emph{right-scattered}, \emph{left-dense} and
\emph{left-scattered} if $\sigma(t)=t$, $\sigma(t)>t$, $\rho(t)=t$
and $\rho(t)<t$, respectively.

Throughout the paper we let $\mathbb{T}=[a,b]\cap\mathbb{T}_{0}$
with $a<b$ and $\mathbb{T}_0$ a time scale
containing $a$ and $b$. 

\begin{remark}
The time scales $\mathbb{T}$ considered in this work have a maximum $b$ and, by definition, $\sigma(b) = b$. For example, let $[a,b] = [1,5]$ and $\mathbb{T}_0 = \mathbb{N}$: in this case $\mathbb{T}=[1,5]\cap\mathbb{N} = \left\{1,2,3,4,5\right\}$
and one has $\sigma(t) = t+1$, $t \in \mathbb{T} \backslash \{5\}$, $\sigma(5) = 5$.
\end{remark}

Following \cite[pp.~2 and 11]{livro}, we define $\mathbb{T}^k=\mathbb{T}\backslash(\rho(b),b]$,
$\mathbb{T}^{k^2}=\left(\mathbb{T}^k\right)^k$ and, more generally,
$\mathbb{T}^{k^n}=\left(\mathbb{T}^{k^{n-1}}\right)^k$, for
$n\in\mathbb{N}$. The following standard notation is used for
$\sigma$ (and $\rho$): $\sigma^0(t) = t$, $\sigma^n(t) = (\sigma
\circ \sigma^{n-1})(t)$, $n \in \mathbb{N}$.

The \emph{graininess function}
$\mu:\mathbb{T}\rightarrow[0,\infty)$ is defined by
$$\mu(t)=\sigma(t)-t,\mbox{ for all $t\in\mathbb{T}$}.$$

We say that a function $f:\mathbb{T}\rightarrow\mathbb{R}$ is
\emph{delta differentiable} at $t\in\mathbb{T}^k$ if there is a
number $f^{\Delta}(t)$ such that for all $\varepsilon>0$ there
exists a neighborhood $U$ of $t$ (\textrm{i.e.}
$U=(t-\delta,t+\delta)\cap\mathbb{T}$ for some $\delta>0$) such
that
$$|f(\sigma(t))-f(s)-f^{\Delta}(t)(\sigma(t)-s)|
\leq\varepsilon|\sigma(t)-s|,\mbox{ for all $s\in U$}.$$
We call $f^{\Delta}(t)$ the \emph{delta derivative} of $f$ at $t$.

If $f$ is continuous at $t$ and $t$ is right-scattered, then (see
Theorem~1.16 (ii) of \cite{livro})
\begin{equation}
\label{derdiscr} f^\Delta(t)=\frac{f(\sigma(t))-f(t)}{\mu(t)}.
\end{equation}

Now, we define the $r^{th}-$\emph{delta derivative}
($r\in\mathbb{N}$) of $f$ to be the function
$f^{\Delta^r}:\mathbb{T}^{k^r}\rightarrow\mathbb{R}$, provided
$f^{\Delta^{r-1}}$ is delta differentiable on $\mathbb{T}^{k^r}$.

For delta differentiable functions $f$ and $g$, the next formulas
hold:

\begin{equation}
\label{transfor}
\begin{aligned}
f^\sigma(t)&=f(t)+\mu(t)f^\Delta(t),\\
(fg)^\Delta(t)&=f^\Delta(t)g^\sigma(t)+f(t)g^\Delta(t)\\
&=f^\Delta(t)g(t)+f^\sigma(t)g^\Delta(t),
\end{aligned}
\end{equation}
where we abbreviate here and throughout $f\circ\sigma$ by
$f^\sigma$. We will also write $f^{\Delta^\sigma}$ as
$f^{\Delta\sigma}$ and all the possible combinations of exponents
of $\sigma$ and $\Delta$ will be clear from the context.

The following lemma will be useful for our purposes.

\begin{lemma}
\label{lem0} Let $t\in \mathbb{T}^k$ ($t\neq\min\mathbb{T}$)
satisfy the property $\rho(t)=t<\sigma(t)$. Then, the jump
operator $\sigma$ is not delta differentiable at $t$.
\end{lemma}
\begin{proof}
We begin to prove that $\lim_{s\rightarrow t^-}\sigma(s)=t$. Let
$\varepsilon>0$ and take $\delta=\varepsilon$. Then for all
$s\in(t-\delta,t)$ we have
$|\sigma(s)-t|\leq|s-t|<\delta=\varepsilon$. Since $\sigma(t)>t$,
this implies that $\sigma$ is not continuous at $t$, hence not
delta-differentiable by Theorem~1.16 (i) of \cite{livro}.
\hfill $\Box$
\end{proof}

A function $f:\mathbb{T}\rightarrow\mathbb{R}$ is called
\emph{rd-continuous} if it is continuous in right-dense points and
if its left-sided limit exists in left-dense points. We denote the
set of all rd-continuous functions by C$_{\textrm{rd}}$ and the
set of all differentiable functions with rd-continuous derivative
by C$_{\textrm{rd}}^1$.

It is known that rd-continuous functions possess an
\emph{antiderivative}, \textrm{i.e.} there exists a function $F$
with $F^\Delta=f$, and in this case an \emph{integral} is defined
by $\int_{a}^{b}f(t)\Delta t=F(b)-F(a)$. It satisfies

\begin{equation}
\int_t^{\sigma(t)}f(\tau)\Delta\tau=\mu(t)f(t).\label{sigma}
\end{equation}

We now present the integration by parts formulas for the delta
integral:

\begin{lemma} (Theorem~1.77 (v) and (vi) of \cite{livro})
\label{integracao partes}
If $a,b\in\mathbb{T}$ and $f,g\in$C$_{\textrm{rd}}^1$, then
\begin{enumerate}

\item$\int_{a}^{b}f(\sigma(t))g^{\Delta}(t)\Delta t
=\left[(fg)(t)\right]_{t=a}^{t=b}-\int_{a}^{b}f^{\Delta}(t)g(t)\Delta
t$;

\item $\int_{a}^{b}f(t)g^{\Delta}(t)\Delta t
=\left[(fg)(t)\right]_{t=a}^{t=b}-\int_{a}^{b}
f^{\Delta}(t)g(\sigma(t))\Delta t$.

\end{enumerate}
\end{lemma}

The main result of the calculus of variations on time scales for
problem \eqref{eq:EL:B} is given by the following necessary
optimality condition.

\begin{theorem} (\cite{CD:Bohner:2004})
\label{Th:B:EL-CV} If $y_\ast \in C_{\textrm{rd}}^1$ is a weak
local minimum of the problem
$$\mathcal{L}[y(\cdot)]=\int_{a}^{b}L(t,y^\sigma(t),y^\Delta(t))\Delta t
\longrightarrow \min ,\mbox{ $y(a)=y_a$, $y(b)=y_b$},$$ then the
Euler-Lagrange equation
\begin{equation*}
L_{y^\Delta}^\Delta(t,y^\sigma_\ast(t),y_\ast^\Delta(t))
=L_{y^\sigma}(t,y^\sigma_\ast(t),y_\ast^\Delta(t)),
\end{equation*}
$t\in\mathbb{T}^{k^2}$, holds.
\end{theorem}

\begin{remark}
In Theorem~\ref{Th:B:EL-CV}, and in what follows, the notation conforms to
that used in \cite{CD:Bohner:2004}. Expression $L_{y^\Delta}^\Delta$ denotes the $\Delta$ derivative of a composition.
\end{remark}

We will assume from now on that the time scale $\mathbb{T}$ has
sufficiently many points in order for all the calculations to make
sense (with respect to this, we remark that
Theorem~\ref{Th:B:EL-CV} makes only sense if we are assuming a
time scale $\mathbb{T}$ with at least three points). Further, we
consider time scales such that:

\begin{description}
\item[(H)] $\sigma(t)=a_1t+a_0$ for some $a_1\in\mathbb{R}^+$ and
$a_0\in\mathbb{R}$, $t \in [a,b)$.
\end{description}

Under hypothesis (H) we have, among others, the differential
calculus ($\mathbb{T}_0=\mathbb{R}$, $a_1 = 1$, $a_0 = 0$), the difference calculus
($\mathbb{T}_0=\mathbb{Z}$, $a_1 = a_0 = 1$) and the quantum calculus
($\mathbb{T}_0=\{q^k:k\in\mathbb{N}_0\}$, with $q>1$,  $a_1 = q$, $a_0 = 0$).

\begin{remark}
\label{remark0} From assumption (H) it follows by Lemma~\ref{lem0}
that it is not possible to have points which are simultaneously left-dense and
right-scattered. Also points that are
simultaneously left-scattered and right-dense do not occur, since $\sigma$ is strictly increasing.
\end{remark}

\begin{lemma}
Under hypothesis (H), if $f$ is a two times delta differentiable
function, then the next formula holds:
\begin{equation}
\label{rui-1} f^{\sigma\Delta}(t)=a_1 f^{\Delta\sigma}(t),\
t\in\mathbb{T}^{k^2}.
\end{equation}
\end{lemma}
\begin{proof}
We have
$f^{\sigma\Delta}(t)=\left[f(t)+\mu(t)f^\Delta(t)\right]^\Delta$
by formula (\ref{transfor}). By the hypothesis on $\sigma$, $\mu$
is delta differentiable, hence
$\left[f(t)+\mu(t)f^\Delta(t)\right]^\Delta
=f^\Delta(t)+\mu^\Delta(t)f^{\Delta\sigma}(t)+\mu(t)f^{\Delta^2}(t)$
and applying again formula (\ref{transfor}) we obtain
$f^\Delta(t)+\mu^\Delta(t)f^{\Delta\sigma}(t)+\mu(t)f^{\Delta^2}(t)
=f^{\Delta\sigma}(t)+\mu^\Delta(t)f^{\Delta\sigma}(t)
=(1+\mu^\Delta(t))f^{\Delta\sigma}(t)$. Now we only need to
observe that $\mu^\Delta(t)=\sigma^\Delta(t)-1$ and the result
follows.
\hfill $\Box$
\end{proof}


\section{Main Results}
\label{secprinci}

Assume that the Lagrangian $L(t,u_0,u_1,\ldots,u_r)$ of problem
(\ref{problema }) has (standard) partial derivatives with respect
to $u_0,\ldots,u_r$, $r\geq 1$, and partial delta derivative with
respect to $t$ of order $r+1$. Let $y\in\mathrm{C}^{2r}$, where
$$\mathrm{C}^{2r}=\left\{y:\mathbb{T}\rightarrow\mathbb{R}:
y^{\Delta^{2r}}\ \mbox{is continuous on}\
\mathbb{T}^{k^{2r}}\right\}.$$

We say that $y_\ast\in\mathrm{C}^{2r}$ is a \emph{weak local
minimum} for (\ref{problema }) provided there exists $\delta>0$
such that $\mathcal{L}(y_\ast)\leq\mathcal{L}(y)$ for all
$y\in\mathrm{C}^{2r}$ satisfying the constraints in (\ref{problema
}) and $\|y-y_\ast\|_{r,\infty} < \delta$, where
$$||y||_{r,\infty} := \sum_{i=0}^{r} ||y^{(i)}||_{\infty},$$
with $y^{(i)} = y^{\sigma^i\Delta^{r-i}}$ and $||y||_{\infty}:=
\sup_{t \in\mathbb{T}^{k^r}} |y(t)|$.

\begin{definition}
We say that $\eta\in C^{2r}$ is an admissible variation for
problem (\ref{problema }) if
\begin{align}
\eta(a)=0,& \quad \eta\left(\rho^{r-1}(b)\right)=0 \nonumber\\
&\vdots\nonumber \\
\eta^{\Delta^{r-1}}(a)=0,&\quad
\eta^{\Delta^{r-1}}\left(\rho^{r-1}(b)\right)=0.\nonumber
\end{align}
\end{definition}

For simplicity of presentation, from now on we fix $r=3$.

\begin{lemma}
\label{lem3} Suppose that $f$ is defined on $[a,\rho^6(b)]$ and is
continuous. Then, under hypothesis (H),
$\int_a^{\rho^5(b)}f(t)\eta^{\sigma^3}(t)\Delta t=0$ for every
admissible variation $\eta$ if and only if $f(t)=0$ for all
$t\in[a,\rho^6(b)]$.
\end{lemma}
\begin{proof}
If $f(t)=0$, then the result is obvious.

Now, suppose without loss of generality that there exists
$t_0\in[a,\rho^6(b)]$ such that $f(t_0)>0$. First we consider the
case in which $t_0$ is right-dense, hence left-dense or $t_0=a$
(see Remark~\ref{remark0}). If $t_0=a$, then by the continuity of
$f$ at $t_0$ there exists a $\delta>0$ such that for all
$t\in[t_0,t_0+\delta)$ we have $f(t)>0$. Let us define $\eta$ by
\[ \eta(t) = \left\{ \begin{array}{ll}
(t-t_0)^8(t-t_0-\delta)^8 & \mbox{if $t \in [t_0,t_0+\delta)$};\\
0 & \mbox{otherwise}.\end{array} \right. \] Clearly $\eta$ is a
$C^6$ function and satisfy the requirements of an admissible
variation. But with this definition for $\eta$ we get the
contradiction
$$\int_a^{\rho^5(b)}f(t)\eta^{\sigma^3}(t)\Delta t
=\int_{t_0}^{t_0+\delta}f(t)\eta^{\sigma^3}(t)\Delta t>0.$$
Now, consider the case where $t_0\neq a$. Again, the continuity of
$f$ ensures the existence of a $\delta>0$ such that for all
$t\in(t_0-\delta,t_0+\delta)$ we have $f(t)>0$. Defining $\eta$ by
\[ \eta(t) = \left\{ \begin{array}{ll}
(t-t_0+\delta)^8(t-t_0-\delta)^8 & \mbox{if $t \in (t_0-\delta,t_0+\delta)$};\\
0 & \mbox{otherwise},\end{array} \right. \] and noting that it
satisfy the properties of an admissible variation, we obtain
$$\int_a^{\rho^5(b)}f(t)\eta^{\sigma^3}(t)\Delta t
=\int_{t_0-\delta}^{t_0+\delta}f(t)\eta^{\sigma^3}(t)\Delta t>0,$$
which is again a contradiction.

Assume now that $t_0$ is right-scattered. In view of
Remark~\ref{remark0}, all the points $t$ such that $t\geq t_0$
must be isolated. So, define $\eta$ such that
$\eta^{\sigma^3}(t_0)=1$ and is zero elsewhere. It is easy to see
that $\eta$ satisfies all the requirements of an admissible
variation. Further, using formula (\ref{sigma})
$$\int_a^{\rho^5(b)}f(t)\eta^{\sigma^3}(t)\Delta t
=\int_{t_0}^{\sigma(t_0)}f(t)\eta^{\sigma^3}(t)\Delta t
=\mu(t_0)f(t_0)\eta^{\sigma^3}(t_0)>0,$$
which is a contradiction.
\hfill $\Box$
\end{proof}

\begin{theorem} \label{thm0}
Let the Lagrangian $L(t,u_0,u_1,u_2,u_3)$ satisfy the conditions in the beginning of the section.
On a time scale $\mathbb{T}$ satisfying (H), if $y_\ast$ is a
weak local minimum for the problem of minimizing
$$
\int_{a}^{\rho^{2}(b)}
L\left(t,y^{\sigma^3}(t),y^{\sigma^{2}\Delta}(t),
y^{\sigma\Delta^2}(t),y^{\Delta^{3}}(t)\right)\Delta t
$$ subject to
\begin{align}
y(a)=y_a,& \ y\left(\rho^{2}(b)\right)=y_b,\nonumber\\
y^{\Delta}(a)=y_a^{1},&\
y^{\Delta}\left(\rho^{2}(b)\right)=y_b^{1},\nonumber\\
y^{\Delta^{2}}(a)=y_a^{2},&\
y^{\Delta^{2}}\left(\rho^{2}(b)\right)=y_b^{2},\nonumber
\end{align}
then $y_\ast$ satisfies the Euler-Lagrange equation
$$L_{u_0}(\cdot)-L^\Delta_{u_1}(\cdot)+\frac{1}{a_1}L^{\Delta^2}_{u_2}(\cdot)
-\frac{1}{a_1^3}L^{\Delta^3}_{u_3}(\cdot)=0,\quad t\in[a,\rho^6(b)],$$
where
$(\cdot)=(t,y_\ast^{\sigma^3}(t),y_\ast^{\sigma^2\Delta}(t),
y_\ast^{\sigma\Delta^2}(t),y_\ast^{\Delta^3}(t))$.
\end{theorem}

\begin{proof}
Suppose that $y_\ast$ is a weak local minimum of $\mathcal{L}$.
Let $\eta\in C^6$ be an admissible variation, \textrm{i.e.} $\eta$
is an arbitrary function such that $\eta$, $\eta^\Delta$ and
$\eta^{\Delta^2}$ vanish at $t=a$ and $t=\rho^2(b)$. Define
function $\Phi:\mathbb{R}\rightarrow\mathbb{R}$ by
$\Phi(\varepsilon)=\mathcal{L}(y_\ast+\varepsilon\eta)$.
This function has a minimum at $\varepsilon=0$, so we must have
(see \cite[Theorem~3.2]{CD:Bohner:2004})
\begin{equation}
\label{rui0}
\Phi'(0)=0.
\end{equation}
Differentiating $\Phi$ under the integral sign (we can do this
in virtue of the conditions we imposed on $L$) with respect to
$\varepsilon$ and setting $\varepsilon=0$,
we obtain from (\ref{rui0}) that
\begin{multline}
\label{rui1}
0=\int_a^{\rho^2(b)}\Big\{L_{u_0}(\cdot)\eta^{\sigma^3}(t)
+L_{u_1}(\cdot)\eta^{\sigma^2\Delta}(t)\\
+L_{u_2}(\cdot)\eta^{\sigma\Delta^2}(t)
+L_{u_3}(\cdot)\eta^{\Delta^3}(t)\Big\}\Delta t.
\end{multline}
Since we will delta differentiate $L_{u_i}$, $i=1,2,3$, we rewrite
(\ref{rui1}) in the following form:
\begin{multline}
\label{rui2}
0=\int_a^{\rho^3(b)}\Big\{L_{u_0}(\cdot)\eta^{\sigma^3}(t)
+L_{u_1}(\cdot)\eta^{\sigma^2\Delta}(t) \\
+L_{u_2}(\cdot)\eta^{\sigma\Delta^2}(t)
+L_{u_3}(\cdot)\eta^{\Delta^3}(t)\Big\}\Delta t\\
+\mu(\rho^3(b))\left\{L_{u_0}\eta^{\sigma^3}
+L_{u_1}\eta^{\sigma^2\Delta}+L_{u_2}\eta^{\sigma\Delta^2}
+L_{u_3}\eta^{\Delta^3}\right\}(\rho^3(b)).
\end{multline}
Integrating (\ref{rui2}) by parts gives
\begin{equation}
\label{rui3}
\begin{aligned}
0&=\int_a^{\rho^3(b)}\Big\{L_{u_0}(\cdot)\eta^{\sigma^3}(t)
-L^\Delta_{u_1}(\cdot)\eta^{\sigma^3}(t) \\
&\qquad \qquad \qquad \qquad \qquad
-L^\Delta_{u_2}(\cdot)\eta^{\sigma\Delta\sigma}(t)
-L^\Delta_{u_3}(\cdot)\eta^{\Delta^2\sigma}(t)\Big\}\Delta t\\
\quad
&+\left[L_{u_1}(\cdot)\eta^{\sigma^2}(t)\right]_{t=a}^{t=\rho^3(b)}
+\left[L_{u_2}(\cdot)\eta^{\sigma\Delta}(t)\right]_{t=a}^{t=\rho^3(b)}
+\left[L_{u_3}(\cdot)\eta^{\Delta^2}(t)\right]_{t=a}^{t=\rho^3(b)}\\
\quad
&+\mu(\rho^3(b))\left\{L_{u_0}\eta^{\sigma^3}
+L_{u_1}\eta^{\sigma^2\Delta}+L_{u_2}\eta^{\sigma\Delta^2}
+L_{u_3}\eta^{\Delta^3}\right\}(\rho^3(b)).
\end{aligned}
\end{equation}
Now we show how to simplify (\ref{rui3}). We start by
evaluating $\eta^{\sigma^2}(a)$:
\begin{align}
\eta^{\sigma^2}(a)&=\eta^\sigma(a)+\mu(a)\eta^{\sigma\Delta}(a)\nonumber\\
&=\eta(a)+\mu(a)\eta^\Delta(a)+\mu(a)a_1\eta^{\Delta\sigma}(a)\label{rui6}\\
&=\mu(a)a_1\left(\eta^\Delta(a)+\mu(a)\eta^{\Delta^2}(a)\right)\nonumber\\
=0 \nonumber,
\end{align}
where the last term of (\ref{rui6}) follows from (\ref{rui-1}).
Now, we calculate $\eta^{\sigma\Delta}(a)$. By (\ref{rui-1}) we
have $\eta^{\sigma\Delta}(a)=a_1\eta^{\Delta\sigma}(a)$ and
applying (\ref{transfor}) we obtain
$$a_1\eta^{\Delta\sigma}(a)=a_1\left(\eta^\Delta(a)+\mu(a)\eta^{\Delta^2}(a)\right)=0.$$
Now we turn to analyze what happens at $t=\rho^3(b)$. It is easy
to see that if $b$ is left-dense, then the last terms of (\ref{rui3})
vanish. So suppose that $b$ is left-scattered. Since
$\sigma$ is delta differentiable, by Lemma~\ref{lem0} we cannot
have points which are simultaneously left-dense and
right-scattered. Hence, $\rho(b)$, $\rho^{2}(b)$ and $\rho^3(b)$
are right-scattered points. Now, by hypothesis
$\eta^\Delta(\rho^2(b))=0$, hence we have by (\ref{derdiscr}) that
$$\frac{\eta(\rho(b))-\eta(\rho^2(b))}{\mu(\rho^2(b))}=0.$$
But $\eta(\rho^2(b))=0$, hence $\eta(\rho(b))=0$. Analogously, we
have
$$\eta^{\Delta^2}(\rho^2(b))=0\Leftrightarrow\frac{\eta^\Delta(\rho(b))
-\eta^\Delta(\rho^2(b))}{\mu(\rho^2(b))}=0,$$ from what follows
that $\eta^\Delta(\rho(b))=0$. This last equality implies
$\eta(b)=0$. Applying previous expressions to the last terms
of (\ref{rui3}), we obtain:
\begin{equation}
\eta^{\sigma^2}(\rho^3(b))=\eta(\rho(b))=0,\nonumber
\end{equation}
$$\eta^{\sigma\Delta}(\rho^3(b))=\frac{\eta^{\sigma^2}(\rho^3(b))
-\eta^\sigma(\rho^3(b))}{\mu(\rho^3(b))}=0,$$
$$\eta^{\sigma^3}(\rho^3(b))=\eta(b)=0,$$
$$\eta^{\sigma^2\Delta}(\rho^3(b))=\frac{\eta^{\sigma^3}(\rho^3(b))
-\eta^{\sigma^2}(\rho^3(b))}{\mu(\rho^3(b))}=0,$$
\begin{equation*}
\begin{split}
\eta^{\sigma\Delta^2}(\rho^3(b))&=\frac{\eta^{\sigma\Delta}(\rho^2(b))
-\eta^{\sigma\Delta}(\rho^3(b))}{\mu(\rho^3(b))}\\
&=\frac{\frac{\eta^\sigma(\rho(b))
-\eta^\sigma(\rho^2(b))}{\mu(\rho^2(b))}-\frac{\eta^\sigma(\rho^2(b))
-\eta^\sigma(\rho^3(b))}{{\mu(\rho^3(b))}}}{\mu(\rho^3(b))}\\
&=0.
\end{split}
\end{equation*}
In view of our previous calculations,
\begin{multline*}
\left[L_{u_1}(\cdot)\eta^{\sigma^2}(t)\right]_{t=a}^{t=\rho^3(b)}
+\left[L_{u_2}(\cdot)\eta^{\sigma\Delta}(t)\right]_{t=a}^{t=\rho^3(b)}
+\left[L_{u_3}(\cdot)\eta^{\Delta^2}(t)\right]_{t=a}^{t=\rho^3(b)}\\
+\mu(\rho^3(b))\left\{L_{u_0}\eta^{\sigma^3}
+L_{u_1}\eta^{\sigma^2\Delta}+L_{u_2}\eta^{\sigma\Delta^2}
+L_{u_3}\eta^{\Delta^3}\right\}(\rho^3(b))
\end{multline*}
is reduced to\footnote{In what follows there is some abuse of notation: $L_{u_3}(\rho^3(b))$ denotes $\left.L_{u_3}(\cdot)\right|_{t = \rho^3(b)}$, that is,
we substitute $t$ in $(\cdot)=(t,y_\ast^{\sigma^3}(t),y_\ast^{\sigma^2\Delta}(t),
y_\ast^{\sigma\Delta^2}(t),y_\ast^{\Delta^3}(t))$ by $\rho^3(b)$.}
\begin{equation}
\label{rui7}
L_{u_3}(\rho^3(b))\eta^{\Delta^2}(\rho^3(b))
+\mu(\rho^3(b))L_{u_3}(\rho^3(b))\eta^{\Delta^3}(\rho^3(b)).
\end{equation}
Now note that
$$\eta^{\Delta^2\sigma}(\rho^3(b))=\eta^{\Delta^2}(\rho^3(b))
+\mu(\rho^3(b))\eta^{\Delta^3}(\rho^3(b))$$
and by hypothesis
$\eta^{\Delta^2\sigma}(\rho^3(b))=\eta^{\Delta^2}(\rho^2(b))=0$.
Therefore,
$$\mu(\rho^3(b))\eta^{\Delta^3}(\rho^3(b))=-\eta^{\Delta^2}(\rho^3(b)),$$
from which follows that (\ref{rui7}) must be zero.
We have just simplified (\ref{rui3}) to
\begin{multline}
\label{rui8}
0 = \int_a^{\rho^3(b)}\Big\{L_{u_0}(\cdot)\eta^{\sigma^3}(t)
-L^\Delta_{u_1}(\cdot)\eta^{\sigma^3}(t)\\
-L^\Delta_{u_2}(\cdot)\eta^{\sigma
\Delta\sigma}(t)-L^\Delta_{u_3}(\cdot)\eta^{\Delta^2\sigma}(t)\Big\}\Delta t.
\end{multline}
In order to apply again the integration by parts formula, we must
first make some transformations in $\eta^{\sigma\Delta\sigma}$ and
$\eta^{\Delta^2\sigma}$. By (\ref{rui-1}) we have
\begin{equation}
\label{rui12}
\eta^{\sigma\Delta\sigma}(t)=\frac{1}{a_1}\eta^{\sigma^2\Delta}(t)
\end{equation}
and
\begin{equation}
\label{rui121}
\eta^{\Delta^2\sigma}(t)=\frac{1}{a_1^2}\eta^{\sigma\Delta^2}(t).
\end{equation}
Hence, (\ref{rui8}) becomes
\begin{multline}
\label{rui9}
0 = \int_a^{\rho^3(b)}\Big\{L_{u_0}(\cdot)\eta^{\sigma^3}(t)
-L^\Delta_{u_1}(\cdot)\eta^{\sigma^3}(t)\\
-\frac{1}{a_1}L^\Delta_{u_2}(\cdot)\eta^{\sigma^2\Delta}(t)
-\frac{1}{a_1^2}L^\Delta_{u_3}(\cdot)\eta^{\sigma\Delta^2}(t)\Big\}\Delta t.
\end{multline}
By the same reasoning as before, (\ref{rui9}) is equivalent to
\begin{multline*}
0=\int_a^{\rho^4(b)}\Big\{L_{u_0}(\cdot)\eta^{\sigma^3}(t)
-L^\Delta_{u_1}(\cdot)\eta^{\sigma^3}(t)\\
-\frac{1}{a_1}L^\Delta_{u_2}(\cdot)\eta^{\sigma^2\Delta}(t)
-\frac{1}{a_1^2}L^\Delta_{u_3}(\cdot)\eta^{\sigma\Delta^2}(t)\Big\}\Delta t\\
\quad +\mu(\rho^4(b))\left\{L_{u_0}\eta^{\sigma^3}
-L^\Delta_{u_1}\eta^{\sigma^3}-\frac{1}{a_1}L^\Delta_{u_2}\eta^{\sigma^2\Delta}
-\frac{1}{a_1^2}L^\Delta_{u_3}\eta^{\sigma\Delta^2}\right\}(\rho^4(b))
\end{multline*}
and integrating by parts we obtain
\begin{equation}
\label{del:rui13}
\begin{aligned}
0&=\int_a^{\rho^4(b)}\Big\{L_{u_0}(\cdot)\eta^{\sigma^3}(t)
-L^\Delta_{u_1}(\cdot)\eta^{\sigma^3}(t)\\
&\qquad \qquad \qquad \qquad \qquad  +\frac{1}{a_1}L^{\Delta^2}_{u_2}(\cdot)\eta^{\sigma^3}(t)
+\frac{1}{a_1^2}L^{\Delta^2}_{u_3}(\cdot)\eta^{\sigma\Delta\sigma}(t)\Big\} \Delta t\\
&-\left[\frac{1}{a_1}L_{u_2}^\Delta(\cdot)\eta^{\sigma^2}(t)\right]_{t=a}^{t=\rho^4(b)}
-\left[\frac{1}{a_1^2}L_{u_3}^\Delta(\cdot)\eta^{\sigma\Delta}(t)\right]_{t=a}^{t
=\rho^4(b)}\\
\quad &+\mu(\rho^4(b))\left\{L_{u_0}\eta^{\sigma^3}
-L^\Delta_{u_1}\eta^{\sigma^3}-\frac{1}{a_1}L^\Delta_{u_2}\eta^{\sigma^2\Delta}
-\frac{1}{a_1^2}L^\Delta_{u_3}\eta^{\sigma\Delta^2}\right\}(\rho^4(b)).
\end{aligned}
\end{equation}
Using analogous arguments to those above,
we simplify (\ref{del:rui13}) to
\begin{multline*}
\int_a^{\rho^4(b)}\Big\{L_{u_0}(\cdot)\eta^{\sigma^3}(t)
-L^\Delta_{u_1}(\cdot)\eta^{\sigma^3}(t)\\
+\frac{1}{a_1}L^\Delta_{u_2}(\cdot)\eta^{\sigma^2\Delta}(t)
+\frac{1}{a_1^3}L^{\Delta^2}_{u_3}(\cdot)\eta^{\sigma^2\Delta}(t)\Big\}\Delta t = 0.
\end{multline*}
Calculations as done before lead us to the final
expression
\begin{multline*}
\int_a^{\rho^5(b)}\Big\{L_{u_0}(\cdot)\eta^{\sigma^3}(t)
-L^\Delta_{u_1}(\cdot)\eta^{\sigma^3}(t)\\
+\frac{1}{a_1}L^{\Delta^2}_{u_2}(\cdot)\eta^{\sigma^3}(t)
-\frac{1}{a_1^3}L^{\Delta^3}_{u_3}(\cdot)\eta^{\sigma^3}(t)\Big\}\Delta t = 0,
\end{multline*}
which is equivalent to
\begin{equation}
\label{almostDone}
\int_a^{\rho^5(b)}\left\{L_{u_0}(\cdot)-L^\Delta_{u_1}(\cdot)
+\frac{1}{a_1}L^{\Delta^2}_{u_2}(\cdot)
-\frac{1}{a_1^3}L^{\Delta^3}_{u_3}(\cdot)\right\}\eta^{\sigma^3}(t)\Delta t = 0.
\end{equation}
Applying Lemma~\ref{lem3} to (\ref{almostDone}),
we obtain the Euler-Lagrange equation
$$L_{u_0}(\cdot)-L^\Delta_{u_1}(\cdot)+\frac{1}{a_1}L^{\Delta^2}_{u_2}(\cdot)
-\frac{1}{a_1^3}L^{\Delta^3}_{u_3}(\cdot)=0,\quad t\in[a,\rho^6(b)].$$
\hfill $\Box$
\end{proof}

Following exactly the same steps in the proofs of Lemma~\ref{lem3}
and Theorem~\ref{thm0} for an arbitrary $r\in\mathbb{N}$, one
easily obtains the Euler-Lagrange equation for problem
(\ref{problema }).

\begin{theorem}(Necessary optimality condition for problems of the calculus of
variations with higher-order delta derivatives) On a time scale
$\mathbb{T}$ satisfying hypothesis (H), if $y_\ast$ is a weak
local minimum for problem (\ref{problema }), then $y_\ast$ satisfies the Euler-Lagrange equation
\begin{equation}
\label{eqrder}
\sum_{i=0}^{r}(-1)^i\left(\frac{1}{a_1}\right)^{\frac{(i
-1)i}{2}}L^{\Delta^i}_{u_i}\left(t,y_\ast^{\sigma^r}(t),y_\ast^{\sigma^{r-1}\Delta}(t),
\ldots,y_\ast^{\sigma\Delta^{r-1}}(t),y_\ast^{\Delta^r}(t)\right) = 0 ,
\end{equation}
$t\in[a,\rho^{2r}(b)]$.
\end{theorem}

\begin{remark}
The factor $\left(\frac{1}{a_1}\right)^{\frac{(i-1)i}{2}}$ in
(\ref{eqrder}) comes from the fact that, after each time we apply
the integration by parts formula, we commute successively $\sigma$
with $\Delta$ using (\ref{rui-1}) (see formulas (\ref{rui12}) and
(\ref{rui121})), doing this $\sum_{j=1}^{i-1}j=\frac{(i-1)i}{2}$
times for each of the parcels within the integral.
\end{remark}


\section*{Acknowledgments}

This work is part of the first author's PhD project and was
partially supported by the Control Theory Group (cotg) of the
Centre for Research on Optimization and Control (CEOC), through
the Portuguese Foundation for Science and Technology (FCT),
cofinanced by the European Community Fund FEDER/POCI 2010. The
authors are grateful to Dorota Mozyrska and Ewa Pawluszewicz for
inspiring discussions during the Workshop on Mathematical Control
Theory and Finance, Lisbon, 10-14 April 2007, where some
preliminary results were presented; to an anonymous referee for helpful comments.



\printindex

\end{document}